\documentclass[a4paper,12pt]{amsart}
\usepackage{amsmath,amsthm,latexsym,amssymb,pgf} 
\usepackage{tikz} 
\usetikzlibrary{calc} 
\usepackage{enumerate}

\newtheorem{defi}{Definition}
\newtheorem{theo}{Theorem}
\newtheorem{coro}{Corollary}

\newtheorem{la}{Lemma}

\usepackage[pdftex,a4paper,
citecolor = blue, colorlinks=true,urlcolor=blue]{hyperref}
\title[Bounding the $k$-Steiner Wiener and Wiener-type indices of trees \ldots]{Bounding the $k$-Steiner Wiener and Wiener-type indices of trees in terms of eccentric sequence}

\author{Peter Dankelmann}
\author{Audace A. V. Dossou-Olory}

\thanks{This work is supported by the National Research Foundation of South Africa, grant 118521}

\address{Peter Dankelmann \and Audace A. V. Dossou-Olory \\ Department of Mathematics and Applied Mathematics \\ University of Johannesburg \\ P.O. Box 524, Auckland Park, Johannesburg 2006, South Africa}

\email{audace@aims.ac.za}

\email{pdankelmann@uj.ac.za}


\begin{document}
	
\maketitle \vspace*{-1cm}

\begin{abstract}
The eccentric sequence of a connected graph $G$ is the nondecreasing sequence of 
the eccentricities of its vertices. The Wiener index of $G$ is the sum of the distances
between all unordered pairs of vertices of $G$. The unique trees that minimise the
Wiener index among all trees with a given eccentric sequence were recently determined 
by the present authors. In this paper we show that these results hold not only for
the Wiener index, but for a large class of distance-based topological indices which
we term Wiener-type indices. Particular cases of this class 
include the hyper-Wiener index, the Harary index, the generalised Wiener index 
$W^{\lambda}$ for $\lambda >0$ and $\lambda <0$, and the reciprocal 
complementary Wiener index. Our results imply and unify known bounds on these  
Wiener-type indices for trees of given order and diameter. 

We also present similar results for the $k$-Steiner Wiener index of trees with a given eccentric sequence. The Steiner distance of a set $A\subseteq V(G)$ is the minimum number of edges in a 
subtree of $G$ whose vertex set contains $A$, and the $k$-Steiner Wiener index is the sum of distances of all $k$-element subsets of $V(G)$. As a corollary, we obtain
a sharp lower bound on the $k$-Steiner Wiener index of trees with given order and diameter, and determine in which cases the extremal tree is unique, thereby correcting an error in the literature.  
\end{abstract}

\section{Introduction and notation}
Graphs in this paper are simple and with at least two vertices. If $G$ is a graph, then 
$V(G)$ and $E(G)$ denote its vertex set and edge set, respectively. The 
\emph{Wiener index} $W(G)$ of a connected graph $G$ is defined as
\begin{align*}
W(G)=\sum_{\{u,v\}\in V(G)} d(u,v)\,,
\end{align*}
where $d(u,v)$ is the usual distance between two vertices $u$ and $v$ of $G$. The Wiener 
index was introduced by Wiener~\cite{Wiener1947} in 1947 as a structural descriptor for the molecular graphs of alkanes. 

The Wiener index has been studied extensively in the mathematical literature; see, for example,~\cite{Dobrynin2011,XuSurvery}. Of particular interest are relations between the Wiener index and other distance-based graph invariants. The problem to determine the maximum
Wiener index of a graph with given order and diameter (defined as the largest of the 
distances between vertices), posed by Plesn\'\i k \cite{Ple1984} in 1984, has 
attracted much attention (see \cite{MukVet2014, SunIkiSkrWuk2019, Wag2006}) and has
only recently been solved asymptotically by Cambie \cite{Cam2018}. 
The minimum Wiener index of graphs with given order and diameter was determined 
in \cite{Ple1984}. The corresponding problem for trees was solved in \cite{LiuPan2008}. 
The relationship between Wiener index and radius (defined as the smallest eccentricity
of the vertices of $G$, where the eccentricity $ec(v)$ of a vertex $v$ in $G$ is
defined as $\max_{u \in V(G)}(d(v,u)$), was considered by Cambie \cite{Cam2019},
who asymptotically determined the minimum Wiener index of a graph with given order 
and radius, thus asymptotically confirming a conjecture by Chen, Wu and An~\cite{CheWuAn2013}. An upper bound on the Wiener index of graphs in terms of
order and radius was given in~\cite{DasNad2017}. 
The relationship between Wiener index and eccentricities was explored
by Darabi, Alizadeh, Klav\v{z}ar and Das~\cite{DarAliKlaDas2018}, who, among 
other results, gave an upper bound on the Wiener index with given total eccentricity
(i.e. the sum of the eccentricities of all vertices) and a lower bound on the 
Wiener index in terms of total eccentricity, order and size. 

\medskip
In~\cite{AudacePeter} the present authors explored the relationship 
between eccentricities by proving a sharp lower bound on the Wiener index
of a tree with a given \emph{eccentric sequence}. The eccentric
sequence of a graph is the nondecreasing sequence of the eccentricities of 
the vertices of $G$. Moreover, the unique extremal tree was determined. 
These results parallel similar theorems on the Wiener-type index as well as the $k$-Steiner Wiener index of trees with a given degree sequence~\cite{SmuckType,JiZhang2019}. 

The aim of this paper is to show that the bound from~\cite{AudacePeter} on the 
Wiener index of 
trees with a given eccentric sequence extends to a much larger class of distance-based topological indices. A large number of topological indices were 
conceived for the purpose of describing relationships between structural formulas 
and molecular graphs, and many of the distance-based topological indices are 
variants or generalisations 
of the Wiener index; see~\cite{diudea1998wiener} and the recent survey~\cite{XuSurvery} for 
more information on these indices and chemical applications. 
In this paper we 
are concerned with variants of the Wiener index that can be expressed
in the form 
\[ W(G;g):=\sum_{\{u,v\}\in V(G)} g(d(u,v)), \]
where $g(x)$ is a nonnegative real-valued function on $\mathbb{N}$ (the set of all positive integers) that is either
nondecreasing or nonincreasing in $x$. We say that $W(G;g)$ is the \emph{Wiener-type} index of $G$ with respect to $g$. 
In addition to the ordinary Wiener index $W(G)$, this definition encompasses some 
well-known distance-based topological indices, such as 
the Harary index $H(G):=\sum_{\{u,v\}\subset V(G)} \frac{1}{d(u,v)}$, 
the hyper-Wiener index $WW(G):=\sum_{\{u,v\}\subset V(G)} {1+ d(u,v) \choose 2}$
(introduced by Randi\'c~\cite{Randic1993} and Klein, Lukovits, and 
Gutman~\cite{Klein1995}; see also \cite{GutmanAl1997,xu2014survey}), 
the generalised Wiener index $W^{\lambda}(G):=\sum_{\{u,v\}\subset V(G)} d(u,v)^{\lambda}$ 
where $\lambda \in \mathbb{R}$, and the reciprocal complementary Wiener index  
$RCW(G):=\sum_{\{u,v\}\subset V(G)} \frac{1}{d+1-d(u,v)}$, where $d$ is the diameter of $G$. 
We determine trees that minimise or maximise the above topological indices among all trees with a given eccentric sequence. 

We also obtain similar results for another generalisation of the Wiener
index, the \emph{$k$-Steiner Wiener} index $SW_k(G)$, defined as
\begin{align*}
SW_k(G)=\sum_{\substack{A\subseteq V(G) \\ |A|=k}} d(A)\,,
\end{align*}
where $d(A)$ denotes the \emph{Steiner distance} of $A$, i.e. the minimum size of 
a connected subgraph of $G$ whose vertex set contains $A$. 
The $k$-Steiner Wiener index was introduced by Li, Mao and Gutman \cite{LikSteiner2016}, 
but the closely related Steiner average distance, defined as ${|V(G)| \choose k} W_k(G)$,
had already been studied in \cite{DankelmannAvSteiner1996, DankelmannAvSteiner1997}.
For further results on the $k$-Steiner Wiener index see, for example, \cite{LikInverseSteiner2018, LUAll2018, JiZhang2019}. 

Our main results imply sharp lower or upper bounds on the above parameters for trees with given order and diameter, together
with a characterisation of the extremal trees.

\medskip
The rest of the paper is organised as follows. In Section \ref{Sec:Preliminaries} we 
present a tree modification that leaves the eccentric sequence of a tree unchanged.
This modification is central to the proofs of our main theorems. In Section~\ref{Sec:caterpillars} we obtain a lower or upper bound for $W(T;g)$ in terms of the eccentric sequence of a tree $T$, and also characterise cases of equality. 
Some corollaries on distance-based topological indices conclude this section. 
In Section~\ref{Sec:KSteinerW}, a lower bound on the $k$-Steiner Wiener index 
of trees with a given eccentric sequence is proved,  and we determine when the extremal
tree is unique. 
In Section~\ref{Sec:Diamet} we use results from the preceding sections to 
derive sharp bounds from the literature  
on the hyper- and generalised Wiener indices, and the Harary index with given order and diameter. We also obtain a known lower bound on the $k$-Steiner Wiener index, 
and determine when the extremal tree is unique, thereby correcting an error 
in the literature. 

%

\medskip

The notation we use is as follows. 
Let $T$ be a tree. By $N_T(v)$ (or simply $N(v)$), we mean the set of all neighbours of vertex $v$ 
in $T$. The path between two vertices $u$ and $w$ in $T$ will be called the $u-w$ path. A 
vertex of degree $1$ in $T$ is called a pendent vertex (or a leaf) of $T$. A pendent edge of 
$T$ is an edge incident with a pendent vertex. If $T_1$ and $T_2$ are trees, then we write $T_1=T_2$ to mean that $T_1$ and $T_2$ are isomorphic.
 
A tree is called a \emph{caterpillar} if a path $P$ remains when all leaves are deleted; 
this path $P$ is called the \emph{backbone} of the caterpillar.

\section{Preliminaries - eccentric sequences of trees}\label{Sec:Preliminaries} 

A sequence of positive integers is called a \emph{tree eccentric sequence} if it is the eccentric sequence of some tree. 
The study of eccentric sequences was initiated by Lesniak~\cite{Lesniak1975}, who 
also provided the following characterisation of tree eccentric sequences. 

\begin{theo}[\cite{Lesniak1975}]\label{LesniakTheo}
A nondecreasing sequence $a_1,a_2,\ldots,a_n$ of $n>2$ positive integers is a tree eccentric sequence if and only if
\begin{enumerate}[i)]
\item $a_1=a_n/2$ and $a_2\neq a_1$, or $a_1=a_2=(a_n+1)/2$ and $a_3\neq a_2$, 
\item for every $a_1<k \leq a_n$, there is $j\in \{2,3,\ldots, n-1\}$ such that $a_j=a_{j+1}=k$.
\end{enumerate}
\end{theo}

\medskip
Since there may be many vertices with the same eccentricity, we denote by 
$b_1,b_2,\ldots,b_l$ the different eccentricities in a tree $T$ in 
increasing order, and by $m_1,m_2,\ldots,m_l$ the number of vertices whose eccentricity is 
$b_1,b_2,\ldots,b_l$, respectively. 
By Theorem~\ref{LesniakTheo}, $b_1,b_2,\ldots,b_l$ are consecutive positive integers and
$m_1\in \{1,2\}$. In particular, the values of $b_1,m_2,\ldots,m_l$ completely determine the eccentric sequence of $T$: its radius is $b_1$, the diameter is $b_1+l-1$ and
so $l=\lceil \frac{d}{2}\rceil +1$ if $d$ is the diameter. Moreover, the number of
centre vertices, $m_1$, is either $1$ (if 
$b_1+l-1$ is even) or $2$ (if $b_1+l-1$ is odd). Hence, we usually write 
$S:=(b_1; m_2,m_3,\ldots,m_l)$ for a tree eccentric sequence. Given a tree 
eccentric sequence $S$,
we write $n_S$, $r_S$ and $d_S$, respectively, for $m_1+m_2+\cdots +m_l$, $b_1$, 
and $b_{l}$. In other words,  $n_S$, $r_S$ and $d_S$ are the order, radius and diameter of a tree realising $S$, respectively. Where there is no danger of confusion, we drop the subscript $S$. Throughout, we assume that $n>2$. 

\begin{defi}
For integers $b_1>0$ and $m_2,\ldots,m_l>1$, we define $\mathbf{T}(b_1; m_2, m_3,\ldots,m_l)$ as the tree obtained from the path $P:=v_0,v_1,\ldots,v_{b_1+l-1}$ by attaching $m_j-2$ pendent edges to vertex $v_{l+1-j}$ for all $2\leq j \leq l$.
\end{defi}

  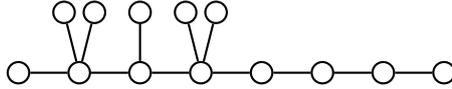
\begin{figure}[h]
  \begin{center}
\begin{tikzpicture}
  [scale=0.4,inner sep=1mm, 
   vertex/.style={circle,thick,draw}, 
   thickedge/.style={line width=2pt}] 
    \node[vertex] (a1) at (0,0) [fill=white] {};
    \node[vertex] (a2) at (2,0) [fill=white] {};
    \node[vertex] (a3) at (4,0) [fill=white] {};
    \node[vertex] (a4) at (6,0) [fill=white] {};
    \node[vertex] (a5) at (8,0) [fill=white] {};
    \node[vertex] (a6) at (10,0) [fill=white] {};
    \node[vertex] (a7) at (12,0) [fill=white] {};
    \node[vertex] (a8) at (14,0) [fill=white] {};    
    \node[vertex] (b1) at (1.5,2) [fill=white] {};
    \node[vertex] (b2) at (2.5,2) [fill=white] {};
    \node[vertex] (b3) at (4,2) [fill=white] {};
    \node[vertex] (b4) at (5.5,2) [fill=white] {};
    \node[vertex] (b5) at (6.5,2) [fill=white] {};

   \draw[thick,black] (a1)--(a2)--(a3)--(a4)--(a5)--(a6)--(a7)--(a8);
    \draw[thick, black] (b1)--(a2)  (b2)--(a2)  
                      (b3)--(a3)   (b4)--(a4)   (b5)--(a4);            
\end{tikzpicture}
\end{center}
\caption{The tree $\mathbf{T}(4;4,3,4)$.}
\label{fig:standard-tree}
\end{figure}

\begin{defi}
For a tree eccentric sequence $S$, we denote the set of all trees with eccentric 
sequence $S$ by $\mathcal{T}_S$. The set of all caterpillars with eccentric
sequence $S$ is denoted by $\mathcal{C}_S$. 
\end{defi}

The following tree modification was already used in~\cite{AudacePeter}. 

\begin{defi}\label{StandDecomp}
Let $T$ be a tree with diameter $d$, that is not a caterpillar. Denote by $P=v_0,v_1,\ldots,v_d$ a longest path in $T$. Then $P$ contains a vertex $v_j$ that has a non-leaf neighbour $u$ not on $P$. We fix $P$ and $u$. Without loss of generality, assume that $j \geq \frac{1}{2}d$. Denote by $U$ the set of vertices that are in a component of $T-u$ not containing $P$. Let $L$ and $R$ be the set of those vertices in $V(T)-U$ that are in the component of $T-v_{j}v_{j+1}$ containing $v_j$ and $v_{j+1}$, respectively. 

Define $T'$ as the tree constructed from $T$ by replacing the edge $uz$ by the edge 
$v_{j+1}z$ for every neighbour $z$ of $u$ in $U$. The tree $T'$ will be referred to as the mate of $T$ with respect to $P$ and $u$, or, if $P$ and $u$ are clear from the context, as the 
mate of $T$.
\end{defi}

It is well-known that in a tree $T$, the eccentricity of a vertex $v$ is the maximum of 
$d_T(u,v)$ and $d_T(v,w)$, where $u-w$ is a longest path in $T$ (see~\cite[Lemma~1]{Lesniak1975}). Using this fact, it is easy to see that for a tree $T$ that is not a
caterpillar, the eccentricity of $v$ in $T$ equals its eccentricity in its mate $T'$. 
Hence we have the following lemma:

\begin{la}[\cite{AudacePeter}] \label{TTprimSameS}
Let $T$ be a tree that is not a caterpillar. If $T'$ is the mate of $T$, then $T$ and $T'$ have the same eccentric sequence. 
\end{la}

\section{Maximising or minimising Wiener-type indices of trees with a given eccentric sequence}\label{Sec:caterpillars}

Our main theorem of this section reads:

\begin{theo}\label{Theo:MinGenr} 
Let $T$ be a tree with eccentric sequence $S=(r;m_2,m_3,\ldots,m_l)$. \\
(a) Let  $g: \mathbb{N} \rightarrow \mathbb{R}$ be nonnegative and nondecreasing.
Then  
\[ W(T;g)\geq W(\mathbf{T}(r; m_2, m_3,\ldots,m_l);g). \] 
If, in addition, $g$ is strictly increasing, then equality holds if and only if 
$T=\mathbf{T}(r; m_2,\ldots,m_l)$. \\[1mm] 
(b) Let  $g: \mathbb{N} \rightarrow \mathbb{R}$ be nonnegative and nonincreasing.
Then  
\[ W(T;g)\leq W(\mathbf{T}(r; m_2, m_3,\ldots,m_l);g). \] 
If, in addition, $g$ is strictly decreasing, then equality holds if and only if 
$T = \mathbf{T}(r; m_2,\ldots,m_l)$. 
\end{theo}

The proof of Theorem 2 proceeds in two steps. In the first step, Lemma 2, we show that among the trees $T \in \mathcal{T}_S$ that minimise (or maximise) $W(T;g)$ there is always a caterpillar. In the second step, we prove that among all caterpillars $T \in \mathcal{C}_S$, the caterpillar $\mathbf{T}(r; m_2,\ldots,m_l)$ minimises or maximises $W(T;g)$. 

\begin{la}\label{Lem:GeneralisedW}
Let $S$ be a tree eccentric sequence and $g: \mathbb{N} \rightarrow \mathbb{R}$ a nonnegative function. \\[1mm]  
(a) Let $g$ be nondecreasing and ${\mathcal T}^{min}_S$ the set of all trees that minimise $W(T,g)$ among all trees $T \in {\mathcal T}_S$. Then the set 
${\mathcal T}^{min}_S$ contains a caterpillar. Moreover, if $g$ is strictly increasing, then ${\mathcal T}^{min}_S$ contains only caterpillars. \\[1mm]  
(b) Let $g$ be nonincreasing and ${\mathcal T}^{max}_S$ the set of all trees that maximise $W(T,g)$ among all trees $T \in {\mathcal T}_S$. Then the set 
${\mathcal T}^{max}_S$ contains a caterpillar. Moreover,  if $g$ is strictly decreasing, then ${\mathcal T}^{max}_S$ contains only caterpillars. 
\end{la}

\begin{proof}
We only prove part (a) since the proof of part (b) is almost identical. \\
If the set $\mathcal{T}_S^{min}$ contains only caterpillars, then there is nothing to 
prove. So assume that the set $\mathcal{T}_S^{min}$ contains a tree $T$ 
that is not a caterpillar. Following Definition~\ref{StandDecomp}, fix $P$, $v_j$ and 
$u$, and let $T'$ be the mate of $T$, and let $L, R, U$ be as defined there. 
By Lemma~\ref{TTprimSameS} we have $T'\in \mathcal{T}_S$. For $a,b \in V(T)$, it is easy to see that $d_{T'}(a,b) \neq d_{T}(a,b)$ only if  $a\in U$ and $b \in R\cup \{u\}$ (or vice versa), and that
\[ d_{T'}(a,b) = \left\{ \begin{array}{cc}
d_{T}(a,b) - 2 & \textrm{if $a\in U$ and $b \in R$,} \\
d_{T}(a,b) + 2 & \textrm{if $a\in U$ and $b=u$.} 
\end{array} \right. \] 

Note that for each $a\in U$, the sets $\{d_T(a,u), d_T(a,v_{j+1}) \}$ 
and  $\{d_{T'}(a,u), d_{T'}(a,v_{j+1}) \}$ coincide.  Hence
\begin{eqnarray}
W(T';g)- W(T;g) 
  & =& 
    \sum_{a \in U, \ b \in R-\{v_{j+1}\} } g(d_{T'}(a,b)) -  g(d_{T}(a,b)) 
     \nonumber \\
 & = &   
    \sum_{a \in U, \ b \in R-\{v_{j+1}\} } g(d_{T}(a,b)-2) -  g(d_{T}(a,b))  
    \nonumber\\
&  \leq &   0,   \label{eq:compare-T-and-T'} 
\end{eqnarray} 
with the last inequality holding since $g$ is nondecreasing. 
It follows that $W(T';g)- W(T;g)\leq 0$, and since 
$W(T;g)$ was minimum, we conclude that $W(T';g)= W(T;g)$. 
Note that $T'$ has more leaves than $T$. Hence, after a finite number of iterations 
of this operation we obtain a caterpillar $C$ with $W(C;g)=W(T;g)$, and thus  
$C \in \mathcal{T}_S^{min}$. This proves the first part of (a).

To prove the second part of (a) note that for $g$ strictly increasing, the inequality~\eqref{eq:compare-T-and-T'} is strict and we get $W(T';g) < W(T;g)$. This contradiction to 
our assumption that $W(T;g)$ is minimum proves that $T$ is a caterpillar. The second part of (a) follows.  
\end{proof}

Lemma \ref{Lem:GeneralisedW} shows that for a proof of a lower bound (if $g$ is nondecreasing) or upper bound (if $g$ is nonincreasing) on $W(T;g)$ for trees $T$ with a given eccentric sequence $S$, we can restrict ourselves to caterpillars. 
Lemma \ref{la:WienerTypehasCater} proves such bounds for caterpillars. The proof 
is a refinement of a method that was developed in~\cite{AudacePeter} for the 
ordinary Wiener index.

\begin{la}\label{la:WienerTypehasCater}
Let $g: \mathbb{N} \rightarrow \mathbb{R}$ be a nonnegative function and $S=(r;m_2,m_3,\ldots,m_l)$ a tree eccentric sequence. \\[1mm]
(a) If $g$ is nondecreasing, then 
\[ W(T;g) \geq W(\mathbf{T}(r; m_2,\ldots,m_l); g) \]
for every caterpillar $T \in \mathcal{C}_S$. 
If, in addition, $g$ is strictly increasing, then equality holds only if
$T = \mathbf{T}(r; m_2,\ldots,m_l)$. \\[1mm]
(b) If $g$ is nonincreasing, then 
\[ W(T;g) \leq W(\mathbf{T}(r; m_2,\ldots,m_l); g) \]
for every caterpillar $T \in \mathcal{C}_S$. 
If, in addition, $g$ is strictly decreasing, then equality holds only if
$T = \mathbf{T}(r; m_2, m_3,\ldots,m_l)$. 
\end{la}

\begin{proof} 
We only give a proof of part (a) since the proof of (b) is almost identical. \\
Let $T \in \mathcal{C}_S$. 
Let the backbone of $T$ be the path $v_1,\ldots,v_{d-1}$, so $d=r+l-1$. 
We may assume that $d>2$ since otherwise $T$ is a star and
$\mathcal{C}_S = \mathcal{T}_S = \{T\}$. 

We fix vertices $v_0 \in N(v_1)-\{v_2\}$ and $v_{d} \in N(v_{d-1})-\{v_{d-2}\}$. 
Then $P_0:=v_0, v_1,\ldots, v_{d}$ is a longest path in $T$. For every 
$j \in \{1,2,\ldots,d-1\}$, let $C_j$ be the set of leaves  adjacent to $v_j$ and not 
on $P_0$. For every $j \in \{1,2,\ldots, \lceil \frac{d-1}{2} \rceil\}$, set 
$D_j=C_j \cup C_{d-j}$. The set $D_j$ contains all vertices of eccentricity $d+1-j$, 
except two vertices that are on $P_0$. Hence $|D_j| = m_{l+1-j}-2$ for all 
$1\leq j \leq \lceil \frac{d-1}{2} \rceil$. Moreover, the sets 
$V(P_0), D_1, D_2,\ldots, D_{\lceil \frac{d-1}{2} \rceil}$ partition $V(T)$.
	
If $X$ and $Y$ are subsets of $V(T)$, then we write $W_T(X)$ for the sum of $g(d_T(x,y))$ 
taken over all $\{x,y\}\subseteq X$, and $W_T(X,Y)$ for the sum of $g(d_T(x,y))$ taken over 
all $x\in X$ and $y\in Y$. With this notation, $W(T;g)$ can be expressed as
\begin{eqnarray}
W(T;g)  & = & W_T(V(P_0))  + \sum_{j=1}^{\lceil (d-1)/2 \rceil} W_T(D_j,V(P_0))
 \nonumber \\
& & + \sum_{j=1}^{\lceil (d-1)/2 \rceil} W_T(D_j)
+ \sum_{1 \leq i < j \leq \lceil (d-1)/2 \rceil} W_T(D_i, D_j)\,. \label{GeneralWiener}
\end{eqnarray}
We now bound the four terms of the right hand side of \eqref{GeneralWiener} separately and
show that each of these terms is minimised by $\mathbf{T}(r; m_2, m_3,\ldots,m_l)$.
 
Clearly,  $W_T(V(P_0)) = W(P_0;g)$ and so 
the first term depends only on $d$ and $g$, but not on the choice of $T$. 

We now consider the second term. Let 
$w\in V(P_0)$. Then $d_T(v,w) =1 +d_T(v_j,w)$ for $v\in C_j$, and 
$d_T(v,w) =1 +d_T(v_{d-j},w)$ for $v\in C_{d-j}$. Moreover for fixed $j$, 
\begin{align*}
\sum_{w \in V(P_0)} g\big((1+d_T(v_j,w)) \big)=\sum_{w \in V(P_0)} g\big ((1+d_T(v_{d-j},w)) \big)
\end{align*}
by symmetry. It follows that
\begin{align*}
W_T(D_j,V(P_0))= |D_j|\sum_{w \in V(P_0)} g\big((1+d_T(v_j,w))\big)\,,
\end{align*}
and so 
\[
\sum_{j=1}^{\lceil (d-1)/2 \rceil} W_T(D_j,V(P_0)) 
= \sum_{j=1}^{\lceil (d-1)/2 \rceil} |D_j|\sum_{w \in V(P_0)} g\big((1+d_{P_0}(v_j,w))\big). 
\]
Since $|D_j| = m_{l+1-j}-2$ for all $1\leq j \leq \lceil \frac{d-1}{2} \rceil$, 
the second term in~\eqref{GeneralWiener} depends only on the sequence $S$ and the 
function $g$, and not on the choice of $T$. 

To bound the third term, note that $d_T(u,v)\geq 2$ for all $u,v\in D_j$ such that 
$u\neq v$. Since $g$ is nondecreasing,
\begin{align}\label{BoundDj}
W_T(D_j) \geq g(2) {|D_j| \choose 2}. 
\end{align}
Equality holds in~\eqref{BoundDj} if all elements of $D_j$ are adjacent to the same 
vertex of $P_0$, i.e. if $C_j=\emptyset$ or $C_{d-j}=\emptyset$. Hence, if 
$T=\mathbf{T}(r; m_2, m_3,\ldots,m_l)$, then equality holds in~\eqref{BoundDj} for all $j$.
On the other hand, if $g$ is strictly increasing, then equality holds in~\eqref{BoundDj} only if $C_j=\emptyset$ or $C_{d-j}=\emptyset$. 

To bound the fourth term, let $v \in D_i$ and $w\in D_j$ where $i<j$. Denote by 
$v'$ (resp. $w'$) the vertex adjacent to $v$ (resp. $w$) in $T$. Then 
$v'\in \{v_i, v_{d-i}\}$, $w'\in \{v_j, v_{d-j}\}$, and $d_T(v,w)=2+d_T(v',w')$. 
Moreover,
\begin{align*}
d_T(v_i,v_j) &= d_T(v_{d-i}, v_{d-j})=j-i\,,\\
d_T(v_i, v_{d-j})&= d_T(v_{d-i}, v_j) = d-i-j \geq j-i
\end{align*}
with equality only if $j=d/2$. It follows that $d_T(v,w) \geq 2+j-i$ with equality in the 
case $j\neq d/2$ only if $v \in C_i$ and $w\in C_j$, or $v\in C_{d-i}$ and $w\in C_{d-j}$. Summing $d_T(v,w)$ over all $v \in D_i$ and $w\in D_j$, and using the 
fact that $g$ is nondecreasing, we obtain
\begin{align}\label{BoundDiDj}
W_T(D_i, D_j) \geq  g((2+j-i)) |D_i| \cdot |D_j|\,,
\end{align}
with equality in the case $j\neq d/2$ if $C_i=C_j=\emptyset,~~\textrm{or}~~C_{d-i} = C_{d-j} = \emptyset$. Hence
\begin{equation} \label{eq:Bound-sum-DiDj}
 \sum_{1 \leq i < j \leq \lceil (d-1)/2 \rceil} W_T(D_i, D_j) 
    \geq   \sum_{1 \leq i < j \leq \lceil (d-1)/2 \rceil} g((2+j-i)) |D_i| \cdot |D_j|. 
\end{equation}
Equality holds in~\eqref{eq:Bound-sum-DiDj} if $C_i=C_j = \emptyset$ or 
$C_{d-i} = C_{d-j}=\emptyset$ for all $1 \leq i < j \leq \lceil (d-1)/2 \rceil$; 
so equality in~\eqref{eq:Bound-sum-DiDj} holds if 
$T=\mathbf{T}(r; m_2, m_3,\ldots,m_l)$.
On the other hand, if $g$ is strictly increasing, then equality holds in~\eqref{BoundDiDj} only if either $C_1=C_2=\cdots = C_{\lfloor (d-1)/2 \rfloor } = \emptyset$, or 
$C_{\lceil (d-1)/2 \rceil +1} =C_{\lceil (d-1)/2 \rceil +2} =\cdots = C_{d-1} = \emptyset$,  
or $d=4$. 

From~\eqref{GeneralWiener} and the fact that $\mathbf{T}(r; m_2, m_3,\ldots,m_l)$ minimises each of the four terms on the right hand side of~\eqref{GeneralWiener}, we obtain 
\[ W(T;g) \geq W(\mathbf{T}(r; m_2, m_3,\ldots,m_l); g) \]
as desired. 

In order to prove the second part of (a) assume that $g$ is strictly increasing
and that $W(T;g) \geq W(\mathbf{T}(r; m_2, m_3,\ldots,m_l); g)$. It follows from the above development that either $C_1=C_2=\cdots = C_{\lfloor (d-1)/2 \rfloor } = \emptyset$, or 
$C_{\lceil (d-1)/2 \rceil +1} =C_{\lceil (d-1)/2 \rceil +2} =\cdots = C_{d-1} = \emptyset$. 
In both cases $T$ is isomorphic to $\mathbf{T}(r; m_2,\ldots,m_l)$. 
\end{proof}

The proof of Theorem~\ref{Theo:MinGenr} now follows from Lemma~\ref{Lem:GeneralisedW} and Lemma~\ref{la:WienerTypehasCater}.

\medskip
Suitable choices of the function $g$ now yield the following corollaries:

\begin{coro}
Let $S=(r;m_2,m_3,\ldots,m_l)$ be a tree eccentric sequence. Then the 
tree $\mathbf{T}(r; m_2,\ldots,m_l)$ uniquely minimises 
the Wiener index $W(T)$,
the hyper Wiener index $WW(T)$,
the generalised Wiener index $W^{\lambda}(T)$ for $\lambda > 0$, 
and the reciprocal complementary Wiener index $RCW(T)$ 
among all trees $T$ whose eccentric sequence is $S$.  
\end{coro}

\begin{coro}
Let $S=(r;m_2,m_3,\ldots,m_l)$ be a tree eccentric sequence. 
Then the tree $\mathbf{T}(r; m_2,\ldots,m_l)$ uniquely maximises 
the generalised Wiener index $W^{\lambda}(T)$ for $\lambda <0$, 
and specifically the Harary index $H(T)$,  
among all trees $T$ whose eccentric sequence is $S$.  
\end{coro}

\section{$k$-Steiner Wiener index}\label{Sec:KSteinerW}

This section is concerned with another generalisation of the ordinary Wiener index, which is 
based on a generalisation of the usual distance between two vertices.
Let $G$ be a connected graph and $A$ a nonempty subset of $V(G)$. The \emph{Steiner  
distance} 
of $A$ is defined as the minimum number of edges in a subtree of $G$ whose 
vertex set contains $A$. A tree of minimum size containing $A$ is  
referred  to as a \emph{Steiner tree} for $A$. 
If $G$ is a tree, then for every set  
$A\subseteq V(G)$, there is only one Steiner tree. The notion of Steiner distance in
graphs was introduced in 1989 by Chartrand et al.~\cite{Chartrand1989} as a natural 
generalisation of the ordinary distance between two vertices in a graph. 
Let $k>0$ be an integer. 
The \emph{$k$-Steiner Wiener} index, $SW_k(G)$, of $G$ is defined as
\begin{align*}
SW_k(G)=\sum_{\substack{A\subseteq V(G) \\ |A|=k}} d(A)\,.
\end{align*}
This index was explicitly introduced by Li, Mao and Gutman~\cite{LikSteiner2016}, although it had previously been considered by Dankelmann, Oellermann and Swart~\cite{DankelmannAvSteiner1996,DankelmannAvSteiner1997} in a different but 
equivalent way. 

Many results on $W(G)$ are known to hold also for $SW_k(G)$; see for 
instance~\cite{DankelmannAvSteiner1996,DankelmannAvSteiner1997,LikSteiner2016,LUAll2018}. 
There is also an application of $SW_k(G)$ in chemistry: it was shown 
in~\cite{GutmanFutula2015} that for trees $T$, there exists an optimal value $\beta$ 
such that $W(T)+\beta SW_k(T)$ yields a better approximation of the boiling points 
of alkanes than $W(T)$.

The main result of~\cite{AudacePeter} states that among all trees $T$ with a given eccentric sequence $S=(r;m_2,m_3,\ldots,m_l)$, the tree $\mathbf{T}(r; m_2,\ldots,m_l)$ 
is the unique tree that minimises $W(T)$. In this section we generalise this
result by showing that the same tree minimises also the $k$-Steiner Wiener index,
and that the minimising tree is unique if $k$ is not too large (see Theorem~\ref{MainTheokSW}). 

We begin by considering some small and large values of $k$. 
For a non-trivial $n$-vertex tree $T$ and $e\in E(T)$, we denote by $n_1(e)$ and $n_2(e)$ 
the orders of the two components of $T-e$. 
As noted in~\cite{LikSteiner2016}, the $k$-Steiner Wiener index of $T$ can then be expressed in terms of the values of $n_1(e)$ and $n_2(e)$ by counting how many times $e$ appears as an edge of a Steiner tree as follows: 
\begin{eqnarray}
SW_k(T)
 &=& \sum_{e \in E(T)} \sum_{j=1}^{k-1} \binom{n_1(e)}{j} \binom{n_2(e)}{k-j} 
      \nonumber \\
& = & \sum_{e \in E(T)} \Big( {n\choose k} - {n_1(e) \choose k} 
   - {n_2(e) \choose k} \Big).  \label{FormulakSteiner}
\end{eqnarray}
First consider $SW_3(T)$. Using~\eqref{FormulakSteiner}, it is easy 
to establish that $SW_3(T)=\frac{n-2}{2}{W(T)}$; see~\cite{LikSteiner2016}.
It follows that $\mathbf{T}(r; m_2,\ldots,m_l)$ also uniquely minimises
$SW_3(T)$ among all trees with this eccentric sequence. 
On the other hand, for $k=n$, $W_n(T)=n-1$ does not depend on the choice of $T$. 
For $k=n-1$, it was shown in \cite{LikInverseSteiner2018} that 
\[ W_{n-2}(T) = n(n-1) -p\,, \] where $p$ is the number of leaves of $T$. Since among all trees in $\mathcal{T}_S$, the caterpillars are exactly the trees
maximising the  number of leaves (see \cite{AudacePeter2}), it follows that 
a tree $T\in \mathcal{T}_S$ minimises $W_{n-1}(T)$ if and only if it is a caterpillar, so in this case the extremal tree is not unique in general. 

\medskip
We present the proof of our main result in two steps, with the first step proving 
that a tree with minimum $k$-Steiner Wiener index among all trees with the same eccentric sequence is necessarily a caterpillar.

\begin{la}\label{Lem:TcaterSwk}
Fix a tree eccentric sequence $S$ and an integer $k\in \{2,3,\ldots, n-1\}$. Let 
$T \in \mathcal{T}_S$ such that $SW_k(T)\geq SW_k(T')$ for all 
$T' \in \mathcal{T}_S$. Then $T$ is a caterpillar.
\end{la}

\begin{proof}
Suppose to the contrary that $T$ is not a caterpillar. Let $P_0$ be a longest path in $T$,  
$v_j$ a vertex on $P_0$ that has a neighbour $u$ which is not a leaf. 
Let $L$, $R$ and $U$ be as in Definition \ref{StandDecomp}, and let
$T'$ be the mate of $T$. Then $T'\in \mathcal{T}_S$ by Lemma~\ref{TTprimSameS}.  

Let $A \subset V(T)$ be a $k$-set. 
It is easy to see that the value $d_{T'}(A)-d_T(A)$ depends only on 
which of the sets 
$A \cap R$, $A \cap L-\{u\}$, $A \cap U$ and $A \cap \{u\}$ are non-empty.  
Clearly, if $A \subseteq L \cup R$ or $A \subseteq U$, then 
$d_{T'}(A)-d_T(A)=0$. Considering all other possibilities, we find the values of $d_{T'}(A) -d_T(A)$ as summarised in Table~\ref{Table:diff} below. 

\begin{table}[htbp]\centering
\caption{Cases where $A \cap (L \cup R)\neq \emptyset$ and $A \cap U \neq \emptyset$. }\label{Table:diff} \vspace*{-0.7cm}
\[ \begin{array}{|c|c|c||c|}
\hline
A \cap (L-\{u\}) & A \cap R  & A \cap \{u\} &  d_{T'}(A) - d_T(A)  \\ \hline
 = \emptyset  & = \emptyset  & \neq \emptyset  & 2 \\
 = \emptyset  & \neq \emptyset  & = \emptyset  & -2 \\
 = \emptyset  & \neq \emptyset  & \neq \emptyset & 0 \\
 \neq \emptyset  & = \emptyset  & = \emptyset  & 0 \\
 \neq \emptyset  & = \emptyset  & \neq \emptyset  & 1 \\
  \neq \emptyset  & \neq \emptyset  & = \emptyset  & -1 \\    
 \neq \emptyset  & \neq \emptyset  & \neq \emptyset  & 0    \\ \hline
\end{array}    \]
\end{table}

For $i\in \{-2, -1, 0, 1,2\}$ we define $M_i$ to be the set of all $k$-subsets 
$A$ of $V(T)$ for which $d_{T'}(A) - d_T(A) = i$. Clearly,
\begin{equation} SW_k(T') - SW_k(T) = \sum_{i=-2}^2 i |M_i| 
     = 2\big(|M_{2}|-|M_{-2}|\big) + \big(|M_1|-|M_{-1}|\big).   \label{eq:Steiner-Wiener-1}
\end{equation}      
From Table~\ref{Table:diff} we see that 
\[ M_{-2} = \{ A \subseteq R \cup U \ | \ A \cap R \neq \emptyset \ \textrm{and} \ 
A \cap U \neq \emptyset \}, \] 
\[ M_{2} = \{ A \subseteq  U \cup \{u\} \ | \ A \cap U \neq \emptyset \ \textrm{and} \ 
u \in A \}, \] 
\[ M_{-1} = \{ A \subseteq (L - \{u\}) \cup R \cup U \ | \ A \cap (L-\{u\}) \neq \emptyset 
\ \textrm{and} \ A \cap R \neq \emptyset \ \textrm{and} \ 
A \cap U \neq \emptyset \}, \] 
\[ M_{1} = \{ A \subseteq L \cup U \ | \ u \in A \ \textrm{and} \ A \cap (L-\{u\}) \neq \emptyset \ \textrm{and} \ 
A \cap U \neq \emptyset \}. \] 

Now fix a vertex $r \in R$.  We define a mapping $f_2: M_{2} \rightarrow M_{-2}$ by setting
$f_2(A) = A \cup \{r\} - \{u\}$ for all $A \in M_2$. Clearly, $f_2$ is an injection, 
so $|M_{-2}| \geq |M_2|$. Similarly, we define a mapping $f_1: M_{1} \rightarrow M_{-1}$ by setting
$f_1(A) = A \cup \{r\} - \{u\}$ for all $A \in M_1$. Clearly, $f_1$ is an injection, but 
since $|R|>1$, $f_1$ is  not a surjection. Therefore, $|M_{-1}|>|M_1|$. Hence,
by~\eqref{eq:Steiner-Wiener-1} we conclude that 
\[   SW_k(T') - SW_k(T) < 0. \]
This contradiction to $SW_k(T)$ being minimum proves the lemma. 
\end{proof}

The following lemma, which we give without proof, is needed for the proof of the 
main theorem of this section.

\begin{la} \label{la:maximising-binomial}
Given $t, z, k \in \mathbb{N}$ with $2t \leq z$. 
Then the function $f(x,y) = {x \choose k} + {y \choose k}$ is maximised, subject to $x+y =z$, $x,y \in \mathbb{N}$ and $x,y \geq t$, if $x=t$ and $y=z-t$. \\
If $k \leq z-t$, then $(x,y)=(t, z-t)$ and $(x,y)=(z-t,t)$ are the only
choices for $x$ and $y$ maximising $f$. 
\end{la}

We can now state and prove the main theorem of this section.

\begin{theo}\label{MainTheokSW}
Let $S=(r;m_2,m_3,\ldots,m_l)$ be a tree eccentric sequence, $T \in \mathcal{T}_S$ 
and $k\in \{2,3,\ldots,n-1\}$. Then 
\[ SW_k(T) \geq SW_k(\mathbf{T}(r; m_2,\ldots,m_l)). \]
If $k \leq n - \lceil \frac{d}{2} \rceil$, then equality implies that
$T=\mathbf{T}(r; m_2,\ldots,m_l)$. 
\end{theo}

\begin{proof}
By Lemma~\ref{Lem:TcaterSwk}, it suffices to prove the theorem for caterpillars. Let 
$T\in \mathcal{C}_S$ and let $P: v_1, v_2,\ldots, v_{d-1}$ be the backbone, 
and let $P_0: v_0, v_1,\ldots, v_{d}$ be a longest path in $T$. 
We define a weight function $w$ on the set of edges of $T$ by 
\[ w(e) = {n \choose k} - {n_1(e) \choose k} - {n_2(e) \choose k}, \]
where $n_1(e)$ and $n_2(e)$ are the orders of the two components of $T-e$. By \eqref{FormulakSteiner}, 
\[ SW_k(T) = \sum_{e \in E(T)} w(e). \]
If $e$ is a pendent edge, then the two components of $T-e$ have $1$ and $n-1$ vertices,
respectively. Hence we have
\begin{equation}   \label{eq:weight-of-pendant-edges} 
w(e) = {n \choose k} - {n-1 \choose k} \quad \textrm {for all $e \in E(T)-E(P)$}. 
\end{equation} 

In order to bound $\sum_{e \in E(P)} w(e)$ from below,  
we partition the set $E(P)$ into sets $E_1, E_2,\ldots, E_{\lfloor (d-1)/2 \rfloor}$, where
$E_i = \{v_iv_{i+1}, v_{d-1-i}v_{d-i}\}$ for $i=1,2,\ldots, \lfloor \frac{d-1}{2} \rfloor$.
Note that $E_i$ contains two edges, unless $d$ is odd and $i=\frac{d-1}{2}$, in which
case $E_{(d-1)/2}= \{v_{(d-1)/2} v_{(d+1)/2}\}$.  
First consider $T-E_i$ for $i<\frac{d-1}{2}$. Then $T-E_i$ consists of three 
components $A_i$, $B_i$ and $C_i$, where 
$A_i$ contains the vertices of the $(v_1, v_i)$-segment of $P$ and its neighbours, 
$B_i$ contains the vertices of the $(v_{i+1}, v_{d-1-i})$-segment of $P$ and its neighbours, and
$C_i$ contains the vertices of the $(v_{d-i}, v_{d-1})$-segment of $P$ and its neighbours. 
Denote the cardinalities of $A_i$, $B_i$ and $C_i$ by $a_i$, $b_i$ and $c_i$, respectively. 
The set $B_i$ contains all vertices whose eccentricity is not more than $d-i$, except 
$v_i$ and $v_{d-i}$. Hence, we have 
\[ b_i = -2 + \sum_{j=1}^{l-i} m_j, \]
\begin{equation} \label{eq:value--ai+ci} 
a_i+c_i = n+ 2 - \sum_{j=1}^{l-i} m_j. 
\end{equation}
Since $\{v_0, v_1,\ldots,v_i\} \subseteq A_i$ and 
$\{v_{d-i}, v_{d-i+1},\ldots, v_{d} \} \subseteq C_i$, 
we also have 
\begin{equation} \label{eq:bound-on-ai-ci}
a_i \geq i+1 \quad \textrm{and} \quad c_i \geq i+1.   
\end{equation}
We now bound the total weight of the edges in $E_i$. 
\begin{eqnarray}
\sum_{e \in E_i} w(e) & = & w(v_iv_{i+1}) + w(v_{d-1-i}v_{d-i}) \nonumber \\
   & = & {n \choose k} - {a_i \choose k} - {n-a_i \choose k} 
        + {n \choose k} - {c_i \choose k} - {n-c_i \choose k} \nonumber 
\end{eqnarray}
It follows from Lemma~\ref{la:maximising-binomial} that 
the term ${a_i \choose k} + {c_i \choose k}$ is maximised, subject to~\eqref{eq:value--ai+ci} and~\eqref{eq:bound-on-ai-ci}, if 
$a_i =  n+ 1 - i - \sum_{j=1}^{l-i} m_j$ and $c_i=i+1$. 
Since by~\eqref{eq:value--ai+ci} we have 
$(n-a_i)+(n-c_i)= n- 2 + \sum_{j=1}^{l-i} m_j$, and by \eqref{eq:bound-on-ai-ci}
we have $n-a_i \leq n-i-1$ and $n-c_i \leq n-i-1$, it follows by 
Lemma~\ref{la:maximising-binomial} that
the term ${n-a_i \choose k} + {n-c_i \choose k}$ is maximised if 
$n-a_i =  i-1 + \sum_{j=1}^{l-i} m_j$ and $n-c_i=n-i-1$. 
Hence 
\begin{eqnarray}  
\sum_{e \in E_i} w(e) & \geq & 2 {n \choose k} 
    - { n+ 1 - i - \sum_{j=1}^{l-i} m_j \choose k}    - {i+1 \choose k} \nonumber \\
 & &   - {i-1 + \sum_{j=1}^{l-i} m_j \choose k}    - {n-i-1 \choose k}.    
    \label{eq:bound-on-total-weight-in-Ei} 
\end{eqnarray}
We note that equality holds in \eqref{eq:bound-on-total-weight-in-Ei} if 
$a_i=i+1$ or $c_i=i+1$, that is, if either none of the vertices $v_0,v_1,\ldots,v_i$
has a neighbour not on $P$, or none of the vertices $v_{d-i}, v_{d-i+1},\ldots,v_d$ has a neighbour not on $P$. That means, in particular, that for the tree 
$\mathbf{T}(r; m_2,\ldots,m_l)$ equality holds for all $i$ with 
$1 \leq i < \frac{d-1}{2}$.  \\[1mm]
{\sc Case 1:} $d$ is odd. \\
Then the set $E_{(d-1)/2}$ consists of only one edge, viz 
$v_{(d-1)/2}v_{(d+1)/2}$. Since removing $v_{(d-1)/2}v_{(d+1)/2}$ splits the path 
$P_0$ into two parts with $(d+1)/2$ vertices each,
we have $n_j(v_{(d-1)/2}v_{(d+1)/2}) \geq  \frac{d+1}{2}$ for $j=1,2$, and so,
by Lemma \ref{la:maximising-binomial}, 
\begin{eqnarray} 
w(v_{(d-1)/2}v_{(d+1)/2}) &=& {n \choose k} - {n_1(v_{(d-1)/2}v_{(d+1)/2})  \choose k} 
                    - {n_2 (v_{(d-1)/2}v_{(d+1)/2}) \choose k} \nonumber \\
  & \geq & {n \choose k} - {(d+1)/2 \choose k} 
                    - {n- (d+1)/2 \choose k}, \label{eq:weight-of-middle-edge-of-P}
\end{eqnarray}
with equality if $n_j(v_{(d-1)/2}v_{(d+1)/2}) = \frac{d+1}{2}$ for some
$j \in \{1,2\}$, so equality holds in particular for $\mathbf{T}(r; m_2,\ldots,m_l)$.  \\
Adding \eqref{eq:weight-of-pendant-edges}, \eqref{eq:bound-on-total-weight-in-Ei}
for $i=1,2,\ldots,\frac{d-3}{2}$, and \eqref{eq:weight-of-middle-edge-of-P} 
we obtain
\begin{eqnarray*}
SW_k(T) & = & \sum_{e \in E(T)-E(P)} w(e) 
       + \Big( \sum_{i=1}^{(d-3)/2} \sum_{e \in E_i} w(e) \Big) 
                              + w(v_{(d-1)/2}v_{(d+1)/2}) \\
     & \geq & (n-d+1) \Big[ {n \choose k} - {n-1 \choose k} \Big]
      +  \Big[ \sum_{i=1}^{(d-3)/2} 2 {n \choose k} 
    - { n+ 1 - i - \sum_{j=1}^{l-i} m_j \choose k}   \\
  & &  - {i+1 \choose k} 
    - {i-1 + \sum_{j=1}^{l-i} m_j \choose k}    - {n-i-1 \choose k} \Big] \\
  & &    + \Big[ {n \choose k} - {(d+1)/2  \choose k} 
                    - {n- (d+1)/2 \choose k} \Big] \\
                       & = & SW_k(\mathbf{T}(r; m_2,\ldots,m_l))\,, 
\end{eqnarray*}
with the last equality holding since for $T=\mathbf{T}(r; m_2,\ldots,m_l)$,  
we have equality in \eqref{eq:weight-of-pendant-edges}, 
\eqref{eq:bound-on-total-weight-in-Ei} and \eqref{eq:weight-of-middle-edge-of-P}. 
This proves the first part of the theorem for the case where $d=r+l-1$ is odd.  

For the proof of the second part of the theorem, assume that 
$$SW_k(T) = SW_k(\mathbf{T}(r; m_2,\ldots,m_l))~~ \text{and that} ~~ k \leq n - \frac{d+1}{2}\,.$$ 
Then we have equality in~\eqref{eq:weight-of-middle-edge-of-P}. 
By Lemma \ref{la:maximising-binomial} this implies that 
$n_j(v_{(d-1)/2}v_{(d+1)/2}) = \frac{d+1}{2}$ for some $j \in \{1,2\}$, 
so either the vertices $v_0, v_1,\ldots, v_{(d-1)/2}$ have no neighbour
outside $P_0$ in $T$, or the vertices $v_{(d+1)/2}, v_{(d+3)/2},\ldots, v_{d}$ 
have no neighbour outside $P_0$ in $T$. It is easy to see that this 
proves that $T$ is isomorphic to $\mathbf{T}(r; m_2,\ldots,m_l)$.  \\[1mm]
{\sc Case 2:} $d$ is even. \\
The proof for the case $d$ even is very similar to the case $d$ odd. Adding~\eqref{eq:weight-of-pendant-edges} and~\eqref{eq:bound-on-total-weight-in-Ei}
for $i=1,2,\ldots,\frac{d-2}{2}$ yields that 
$SW_k(T) \geq  SW_k(\mathbf{T}(r; m_2,\ldots,m_l))$, so the first part of the theorem
holds. As in Case 1 we conclude that either the vertices
$v_0, v_1,\ldots, v_{(d-2)/2}$ have no neighbour
outside $P_0$ in $T$, or the vertices $v_{(d+2)/2}, v_{(d+4)/2}, \\ \ldots, v_d$ 
have no neighbour outside $P_0$ in $T$, and so 
$T$ is isomorphic to $\mathbf{T}(r; m_2,\ldots,m_l)$.
\end{proof}

\medskip
We now show that the condition $k \leq n - \lceil \frac{d}{2} \rceil$ for uniqueness
of the extremal tree in Theorem \ref{MainTheokSW} is best possible.
Let $n,d, k \in \mathbb{N}$ with $n \geq d+3$ and 
$n+1 - \lceil \frac{d}{2} \rceil \leq k \leq n-1$
be given. As before, let $P_0$ be the path $v_0, v_1,\ldots,v_d$.\\
First assume that $d$ is odd. 
let $T_1$ be obtained from $P_0$ by adding $n-d-1$ new vertices and joining 
them to $v_{(d-1)/2}$, and 
let $T_2$ be obtained from $P_0$ by adding $n-d-1$ new vertices and joining 
one of these to $v_{(d+1)/2}$ and the remaining $n-d-2$ vertices to $v_{(d-1)/2}$. 
Then both, $T_1$ and $T_2$ have $n-d+1$ vertices of eccentricity $\frac{d+3}{2}$, 
and two vertices of eccentricity $i$ for all 
$i\in \mathbb{N}$ with $\frac{d+1}{2} \leq i \leq d$ and $i \neq \frac{d+3}{2}$.
Clearly, $T_1 = {\bf T}(\frac{d+1}{2}; n-d+1, 2, 2, \ldots,2)$. The trees
$T_1$ and $T_2$ for $n=11$ and $d=7$ are shown in Figure \ref{fig:exmple-for-d-odd},  
where $v_{(d-1)/2}$ and $v_{(d+1)/2}$ are solid grey. 

  \begin{figure}[h]
  \begin{center}
\begin{tikzpicture}
  [scale=0.4,inner sep=1mm, 
   vertex/.style={circle,thick,draw}, 
   thickedge/.style={line width=2pt}] 
    \node[vertex] (a1) at (0,0) [fill=white] {};
    \node[vertex] (a2) at (2,0) [fill=white] {};
    \node[vertex] (a3) at (4,0) [fill=white] {};
    \node[vertex] (a4) at (6,0) [fill=gray] {};
    \node[vertex] (a5) at (8,0) [fill=gray] {};
    \node[vertex] (a6) at (10,0) [fill=white] {};
    \node[vertex] (a7) at (12,0) [fill=white] {};
    \node[vertex] (a8) at (14,0) [fill=white] {};    
    \node[vertex] (b1) at (5,2) [fill=white] {};
    \node[vertex] (b2) at (6,2) [fill=white] {};
    \node[vertex] (b3) at (7,2) [fill=white] {};

   \draw[thick,black] (a1)--(a2)--(a3)--(a4)--(a5)--(a6)--(a7)--(a8);
    \draw[thick, black] (b1)--(a4)  (b2)--(a4)  
                      (b3)--(a4);            
\end{tikzpicture}
\hspace*{2em}
\begin{tikzpicture}
  [scale=0.4,inner sep=1mm, 
   vertex/.style={circle,thick,draw}, 
   thickedge/.style={line width=2pt}] 
    \node[vertex] (a1) at (0,0) [fill=white] {};
    \node[vertex] (a2) at (2,0) [fill=white] {};
    \node[vertex] (a3) at (4,0) [fill=white] {};
    \node[vertex] (a4) at (6,0) [fill=gray] {};
    \node[vertex] (a5) at (8,0) [fill=gray] {};
    \node[vertex] (a6) at (10,0) [fill=white] {};
    \node[vertex] (a7) at (12,0) [fill=white] {};
    \node[vertex] (a8) at (14,0) [fill=white] {};    
    \node[vertex] (b1) at (5,2) [fill=white] {};
    \node[vertex] (b2) at (6,2) [fill=white] {};
    \node[vertex] (b3) at (8,2) [fill=white] {};

   \draw[thick,black] (a1)--(a2)--(a3)--(a4)--(a5)--(a6)--(a7)--(a8);
    \draw[thick, black] (b1)--(a4)  (b2)--(a4)  
                      (b3)--(a5);            
\end{tikzpicture}
\end{center}
\caption{The trees $T_1$ and $T_2$ for $n=11$ and $d=7$.}
\label{fig:exmple-for-d-odd}
\end{figure}
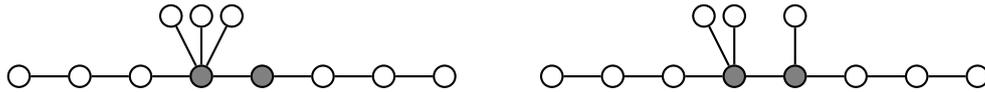
 
To see that $SW_k(T_1)=SW_k(T_2)$, notice that if an edge of $T_1$ splits 
$T_1$ into two components with $a$ and $b$ vertices, then so does the
corresponding edge of $T_2$, unless $v_{(d-1)/2}v_{(d+1)/2}$. 
Therefore, the weight of each edge of $T_1$ except $v_{(d-1)/2}v_{(d+1)/2}$. 
equals the weight of its corresponding 
edge in $T_2$. The edge $v_{(d-1)/2}v_{(d+1)/2}$ also has the same weight 
${n \choose k}$, in $T_1$ and in $T_2$ since $k>n-\frac{d+1}{2}$. By~\eqref{FormulakSteiner} 
we conclude that $SW_k(T_1)=SW_k(T_2)$. Hence the extremal 
tree is not unique for $n+1 - \lceil \frac{d}{2} \rceil \leq k \leq n-1$.

If $d$ is even, then a similar construction demonstrates that the extremal
tree is not unique. 
Let $T_1$ be obtained from $P_0$ by adding $n-d-1$ new vertices and joining 
them to $v_{(d-2)/2}$, and 
let $T_2$ be obtained from $P_0$ by adding $n-d-1$ new vertices and joining 
one of these to $v_{(d+2)/2}$ and the remaining $n-d-2$ vertices to $v_{(d-2)/2}$. 
The same reasoning as above shows that $SW_k(T_1)=SW_k(T_2)$ for 
$n+1 - \lceil \frac{d}{2} \rceil \leq k \leq n-1$, and so the extremal tree
is not unique in this case.

\section{Trees with a given diameter}\label{Sec:Diamet}

In this section we apply our results from the two preceding sections to 
derive sharp lower or upper bounds (depending on whether $g$ is nondecreasing or 
nonincreasing) on the Wiener-type index $W(T;g)$ for trees $T$ with given order and diameter. 
Given $n,d \in \mathbb{N}$ with $2 \leq d \leq n-1$, let $\mathbf{T}_{d,n}$ be 
the tree obtained by attaching $n-d-1$ pendent edges to a centre vertex of 
the path of order $d+1$; so
$\mathbf{T}_{d,n} = \mathbf{T}( \lceil d/2 \rceil; n-d+1,2,2,\ldots,2)$.  
The tree $\mathbf{T}_{d,n}$ has been shown to minimise
or maximise several distance-based topological indices among all trees with
order $n$ and diameter $d$. In this section we show that these results are
consequences of our results from the previous sections, thereby correcting an error in the literature.

\begin{la} \label{la:diameter}
Let $S=(r; m_2,\ldots,m_l)$ be a tree eccentric sequence 
with $d=r+l-1 \geq 3$ and $\max\{m_3, m_4,\ldots,m_l\} >2$. Let 
$i\in \{3,4,\ldots,l\}$ be the largest value such that $m_i>2$. 
Let $S'$ be the sequence 
$S' =(r, m_2', m_3', \ldots , m_l')$ with $m_{i-1}' =m_{i-1}+m_i-2$,  $m_i'=2$
and $m_j'=m_j$ for all $j\notin \{i-1, i\}$. \\[1cm]
(a) Then the sequence $S'$  is tree eccentric. \\
(b) If $g$ is nonnegative and strictly increasing, then  
\[ W( \mathbf{T}( r; m_2', m_3',\ldots,m_l');g) 
              <  W( \mathbf{T}( r; m_1, m_2,\ldots,m_l);g). \]
(c) If $g$ is nonnegative and strictly decreasing, then 
\[ W( \mathbf{T}( r; m_2', m_3',\ldots,m_l');g) 
              >  W( \mathbf{T}( r; m_2, m_3,\ldots,m_l);g). \]
(d) If $2 \leq k \leq n-1$, then 
\[ SW_k(\mathbf{T}( r; m_2', m_3',\ldots,m_l')) 
              \leq  SW_k(\mathbf{T}( r; m_1, m_2,\ldots,m_l)), \]
and if $k\leq n-\lceil d/2 \rceil$ and $i=3$, then the inequality is strict.              
\end{la}

\begin{proof}
Denote the tree  $\mathbf{T}( r; m_2, m_3,\ldots,m_l)$ by $T$ and 
$\mathbf{T}( r; m_2', m_3',\ldots,m_l')$ by $T'$. Since $S'$ is the 
eccentric sequence of $T'$, the sequence $S'$ is tree eccentric. 
This proves part (a). Note that $n_S=n_{S'}=n$ and that $d_S=d_{S'}=d$.

To prove part (b) assume that $g$ is strictly increasing. 
Let $P_0:v_0, v_1,\ldots,v_d$ be a longest path in $T$ as well as $T'$.
Let $p= l+1-i$. So the eccentricity of $v_p$ is $r+i-2$ in $T$, and $v_p$ is therefore the first (from $v_0$) of the vertices of $P_0$ that in $T$ 
has neighbours outside $P_0$. Denote the set $N(v_p)-V(P_0)$ by $A$. 
Then $T'$ can be obtained from $T$ by replacing the edge
$xv_p$ with $xv_{p+1}$ for all $x\in A$.  Denote the set $V(T)-(V(P_0) \cup A)$ by $B$. 

Comparing $W(T;g)$ and $W(T';g)$, we note that the distance between two 
vertices in $T$ differs from the distance between the same vertices in $T'$ if
and only if exactly one of them belongs to $A$. Hence
\begin{eqnarray}
W(T;g) - W(T';g) & = &
    \sum_{x\in A, y \in B} [g(d_{T}(x,y)) - g(d_{T'}(x,y)) ] \nonumber \\
  & &   + \sum_{x\in A} \sum_{j=0}^d [g(d_{T}(x,v_j)) - g(d_{T}(x,v_j)) ]
                    \nonumber \\
  & = &     \sum_{x\in A, y \in B} [g(d_{T}(x,y)) - g(d_{T}(x,y)-1) ] \nonumber \\
  & &   + |A| \sum_{j=0}^d [g(|j-p|+1) - g(|j-(p+1)|+1) ]. \nonumber
\end{eqnarray}
Now
$\sum_{x\in A, y \in B} [g(d_{T}(x,y)) - g(d_{T}(x,y)-1)] > 0$
since $g$ is strictly increasing. In the second sum we add the terms 
$g(p+1), g(p), \ldots,g(1)$ and $g(2), g(3), \ldots, g(d-p+1)$,
and we subtract 
$g(p+2), g(p+1), \ldots,g(1)$ and $g(2), g(3), \ldots, g(d-p)$.
Cancelling equal terms, we obtain 
\[ W(T;g) - W(T';g) \geq |A| \big(g(d-p+1)-g(p+2) \big)  
                > 0 \]
since $d-p+1 \geq p+2$ and $g$ is increasing.  This completes the 
proof of (b). 
We omit the proof of (c) since it is almost identical to the proof of (b).                 

\medskip
To prove part (d) we make use of the fact that 
$SW_k(T) = \sum_{e \in E(T)} w(e)$, where the weight of an edge $e$ is
defined by 
$w(e) = {n \choose k} - {n_1(e) \choose k} - {n_2(e) \choose k}$, and  
$n_1(e)$ and $n_2(e)$ are the orders of the two components of $T-e$. 

It is easy to verify that for every edge $e$ of $T$, the corresponding edge
of $T'$ has the same weight, unless $e=v_{p}v_{p+1}$, where $p$ is as
defined in the proof of (b). Denoting the weight of $e$ in $T$ and $T'$ 
by $w(e)$ and $w'(e)$, respectively, we have 
$w(v_{p}v_{p+1})= {n \choose k} - {p+m_i-1 \choose k} - {n-p-m_i+1 \choose k}$ and 
$w'(v_{p}v_{p+1})= {n \choose k} - {p+1 \choose k} - {n-p-1 \choose k}$. Hence
\begin{eqnarray*} 
SW_k(T) - SW_k(T') & = & w(v_{p}v_{p+1})-w'(v_{p}v_{p+1})  \\
  & \hspace*{-9em}   = & \hspace*{-5em} - {p+m_i -1 \choose k} - {n-p-m_i +1 \choose k} 
              + {p+1 \choose k} + {n-p-1 \choose k}   \\
  & \hspace*{-9em}   \geq  & \hspace*{-5em} 0,                 
\end{eqnarray*}  
with the last inequality holding by Lemma \ref{la:maximising-binomial}
since $(p+1)+(n-p-1)=n=(p+m_i-1)+(n-p-m_i+1)$.
If $k \leq n - \lceil \frac{d}{2} \rceil$ and $i=3$, then $p=l-2$ 
and thus $n-p-1=n+1-l= n - \lceil \frac{d}{2} \rceil \geq k$, and 
by Lemma~\ref{la:maximising-binomial} the inequality is strict, i.e.
$SW_k(T) - SW_k(T')>0$, as desired. 

\end{proof}

\begin{theo}\label{theo:MinGenrDiam}
Let $g(x)$ be a function on $\mathbb{N}$ that is nonnegative.  
Let $T$ be a tree with order $n$ and diameter $d$. \\
(a) If $g$ is strictly increasing, then 
\[ W(T;g)\geq W(\mathbf{T}_{d,n}; g). \]
Equality holds if and only if $T=\mathbf{T}_{d,n}$. \\
(b) If $g$ is strictly decreasing, then 
\[ W(T;g)\leq W(\mathbf{T}_{d,n}; g). \]
Equality holds if and only if $T=\mathbf{T}_{d,n}$.
\end{theo}

\begin{proof}
We only prove part (a) since the proof of (b) is almost identical.  \\
Let $T$ be a tree with order $n$ and diameter $d$, and let 
$S=(r, m_2,\ldots,m_l)$ be its eccentric sequence. 
By Theorem \ref{Theo:MinGenr}, 
\begin{equation} \label{eq:coro-diameter-1}
W(T;g) \geq W(\mathbf{T}(r;m_2,m_3,\ldots,m_l)), 
\end{equation}
with equality if and only if $T=\mathbf{T}(r;m_2,m_3,\ldots,m_l)$. \\
We claim that
\begin{equation} \label{eq:coro-diameter-2} 
W(\mathbf{T}(r;m_2,m_3,\ldots,m_l)) \geq W(\mathbf{T}_{d,n};g). 
\end{equation} 
Indeed, if $\mathbf{T}(r;m_2,m_3,\ldots,m_l) = \mathbf{T}_{d,n}$,
then there is nothing to prove, and if 
$\mathbf{T}(r;m_2,m_3,\ldots,m_l)) \neq \mathbf{T}_{d,n}$
then it is easy to see that $\max\{m_3, m_4,\ldots,m_l\} >2$, and so
repeated application of Lemma~\ref{la:diameter} yields~\eqref{eq:coro-diameter-2}.
Moreover, equality in~\eqref{eq:coro-diameter-2} holds by Lemma~\ref{la:diameter}
only if $\mathbf{T}(r;m_2,m_3,\ldots,m_l) = \mathbf{T}_{d,n}$. \\
Now~\eqref{eq:coro-diameter-1} and~\eqref{eq:coro-diameter-2} yield the 
inequality in part (a) of the theorem. \\
If we have equality, i.e., if $W(T;g) = W(\mathbf{T}_{d,n}; g)$, then
we have equality in \eqref{eq:coro-diameter-1} and \eqref{eq:coro-diameter-2}, 
and thus $T=\mathbf{T}(r;m_2,m_3,\ldots,m_l)=  \mathbf{T}_{d,n}$. 
\end{proof}

We note that the inequality in part (a) of Theorem~\ref{theo:MinGenrDiam}
holds even if $g$ is not strictly increasing but only nondecreasing. However,
in this case equality may hold for trees other than  $\mathbf{T}_{d,n}$. 
The same holds true for part (b) of Theorem \ref{theo:MinGenrDiam}. \\

\medskip
The following corollaries are immediate consequences of Theorem~\ref{theo:MinGenrDiam}. 
Let $d,n$ be fixed integers such that $1 <d \leq n-1$.

\begin{coro}[\cite{xu2014survey}]
Among all trees with order $n$ and diameter $d$, the tree $\mathbf{T}_{d,n}$ is the
unique tree that minimises the hyper-Wiener index.
\end{coro}

\begin{coro}
(a) Let $\lambda \in \mathbb{R}$ with $\lambda > 0$, Among all trees with order $n$
and diameter $d$, the tree  $\mathbf{T}_{d,n}$ is the unique tree that minimises 
the generalised Wiener index $W^{\lambda}(T)$. \\
(b) Let $\lambda \in \mathbb{R}$ with $\lambda<0$. Among all trees with order $n$
and diameter $d$, the tree  $\mathbf{T}_{d,n}$ is the unique tree that maximises 
the generalised Wiener index $W^{\lambda}(T)$. 
\end{coro}

\begin{coro}[\cite{YuFengHarary}]
Among all trees with order $n$ and diameter $d$, the tree $\mathbf{T}_{d,n}$ is the
unique tree that maximises the Harary index.
\end{coro}

\begin{coro}[\cite{CaiZhou2009}]
Among all trees with order $n$ and diameter $d$, the tree $\mathbf{T}_{d,n}$ is the
unique tree that minimises the reciprocal complementary Wiener index.
\end{coro}

\medskip
Theorem~\ref{MainTheokSW} (see Section~\ref{Sec:KSteinerW}) also implies the main result of~\cite{LUAll2018}, a sharp lower bound on  the $k$-Steiner Wiener index of trees with given order and diameter. We note, however, a minor error in \cite{LUAll2018}, where the authors incorrectly claim that 
the extremal tree is unique for all $k \in \{2,3,\ldots,n-2\}$. The trees 
$T_1$ and $T_2$ (Figure~\ref{fig:exmple-for-d-odd}) presented at the end of Section \ref{Sec:KSteinerW} show that
this is not the case. In our corollary below we correct this error.

\begin{coro}[\cite{LUAll2018}] 
Let $T$ be a tree with order $n$ and diameter $d$, and let $k\in \{2,3,\ldots, n-1\}$. 
Then 
\begin{equation} \label{eq:Steiner-Wiener-coro} 
SW_k(T) \geq SW_k(\mathbf{T}_{d,n}). 
\end{equation}
If $k \leq n - \lceil \frac{d}{2} \rceil$, then equality implies that 
$T =\mathbf{T}_{d,n}$. 
\end{coro}

\begin{proof}
Let $T$ be a tree with order $n$ and diameter $d$, and let 
$S=(r, m_2,\ldots,m_l)$ be its eccentric sequence. By Theorem~\ref{MainTheokSW}
we have 
$SW_k(T) \geq SW_k(\mathbf{T}(r; m_2,\ldots,m_l))$. \\[1mm]
{\sc Case 1:}  $\mathbf{T}(r; m_2,\ldots,m_l) = \mathbf{T}_{d,n}$. \\ 
Then~\eqref{eq:Steiner-Wiener-coro} holds. If $k\leq n- \lceil d/2 \rceil$, then by Theorem~\ref{MainTheokSW} we have equality in~\eqref{eq:Steiner-Wiener-coro} only if 
$T=\mathbf{T}(r; m_2,\ldots,m_l)$, i.e., 
if $T = \mathbf{T}_{d,n}$. \\[1mm]
{\sc Case 2:}  $\mathbf{T}(r; m_2,\ldots,m_l) \neq \mathbf{T}_{d,n}$. \\ 
Then $l \geq 3$ and there exists an index $i$ with $m_i \geq 3$. Indeed, if
$l=2$, then $T$ has diameter $2$, but the star is the only tree with that
property. Also, if $m_i=2$ for all $i\geq 3$, then it is easy to see that
$T=\mathbf{T}_{d,n}$, a contradiction. 
Define $i(r; m_2,m_3,\ldots,m_l)$ to be the largest $i \in \{2,3,\ldots,l\}$ 
for which $m_i>2$. Let $S'$ be the sequence 
$S' =(r, m_2', m_3', \ldots , m_l')$ with $m_{i-1}' =m_{i-1}+m_i-2$,  $m_i'=2$
and $m_j'=m_j$ for all $j\notin \{i-1, i\}$, as described in 
Lemma \ref{la:diameter}, and let $T':=\mathbf{T}(r; m_2',\ldots,m_l')$. Clearly,
$i(r; m_2',\ldots,m_l') = i(r; m_2,\ldots,m_l) -1$, and the tree  
$T'$ has diameter $d$. By part (d) of Lemma~\ref{la:diameter} we have 
$SW_k(T) \geq SW_k(T')$. Applying this modification to the sequence
$S' =(r, m_2', m_3', \ldots , m_l')$ we obtain a sequence 
$S'' =(r, m_2'', m_3'', \ldots , m_l'')$ with $i(S'')=i(S')-1$. Letting 
$T'':=\mathbf{T}(r; m_2'',\ldots,m_l'')$ we have, as above, $SW_k(T') \geq SW_k(T'')$.\\ 
Repeating this step  $s:=i(r; m_2,m_3,\ldots,m_l)-2$ times we obtain a sequence
of sequences $S, S', S'',\ldots, S^{(s)}$ with 
$(i(S), i(S'), i(S''),\ldots, i(S^{s)}) = (i(S), i(S)-1, i(S)-2,\ldots,2)$, 
as well as corresponding trees 
$\mathbf{T}(r;m_2, m_3,\ldots,m_l), T', T'',\ldots,T^{(s)}$, with 
$SW_k(\mathbf{T}(r;m_2, m_3,\ldots,m_l)) \geq SW_k(T') \geq SW_k(T'') \geq \cdots \geq SW_k(T^{(s)})$. 
Since $i(S^{(s-1)})=3$, we have that the last inequality in this chain
is strict if $k\leq n-\lceil d/2 \rceil$ (see Lemma~\ref{la:diameter}). It follows that 
$SW_k(T) \geq SW_k(\mathbf{T}(r;m_2, m_3,\ldots,m_l)) >  SW_k(T^{(s)})$. 
It is now easy to see that $T^{(s)} = \mathbf{T}_{d}$ since $i(S^{(s)})=2$. 
Hence the corollary follows. 
\end{proof}

\end{document}